\documentclass[11pt]{amsart}

\usepackage[margin=1in]{geometry}

\usepackage{amsmath, amsfonts, amssymb, amsthm}
\usepackage{amsrefs}
\usepackage{mathrsfs}
\usepackage{enumitem} 
\usepackage{hyperref}
\usepackage{color}
\usepackage{dsfont}
\usepackage{csquotes}
\usepackage{epigraph}
\usepackage{comment}
\usepackage{mlmodern}
\usepackage[T1]{fontenc}

\usepackage{color}
\definecolor{darkblue}{rgb}{0.0,0.0,0.3}
\hypersetup{
    colorlinks=false,
    linkcolor=blue,
    urlcolor=darkblue,
    }

\newtheorem{theorem}{Theorem}[section]
\newtheorem{proposition}[theorem]{Proposition}

\newtheorem{lemma}[theorem]{Lemma}

\theoremstyle{definition}

\newtheorem{claim}[theorem]{Claim}

\theoremstyle{remark}
\newtheorem{remark}[theorem]{Remark}

\numberwithin{equation}{section}

\newcommand{\bR}{\mathbb{R}}
\newcommand{\bH}{\mathbb{H}}

\newcommand{\bS}{\mathbb{S}}

\newcommand{\bn}{\mathbf{n}}

\newcommand{\Ampere}{Amp\`{e}re}

\newcommand{\Garding}{G\r{a}rding}

\DeclareMathOperator{\const}{const}
\DeclareMathOperator{\tr}{tr}

\title[Constant scalar curvature equation in hyperbolic space]{Hypersurfaces of constant scalar curvature in hyperbolic space with prescribed asymptotic boundary at infinity}
\author{Bin Wang}
\address[]{Department of Mathematics, The Chinese University of Hong Kong, Shatin, New Territories, The Hong Kong Special Administrative Region of the People's Republic of China.} 
\email{bwangmath@link.cuhk.edu.hk}

\begin{document}

\begin{abstract}
This article concerns a natural generalization of the classical asymptotic Plateau problem in hyperbolic space. We prove the existence of a smooth complete hypersurface of constant scalar curvature with a prescribed asymptotic boundary at infinity. The desired hypersurface is constructed as the limit of constant scalar curvature graphs (with respect to vertical geodesics) over a fixed compact domain in a horosphere, and the problem is thus reduced to solving a Dirichlet problem for a fully nonlinear elliptic partial differential equation which is degenerate along the boundary. Previously, the result was known only for a restricted range of curvature values. Now, in this article, by introducing some new techniques, we are able to solve the Dirichlet problem for all possible curvature values. The main ingredient is the establishment of the crucial second order a priori estimates for admissible solutions. 

\end{abstract}
\maketitle
\setcounter{tocdepth}{1} 
\tableofcontents

\section{Introduction}
Fix $n \geq 3$. Let $\bH^{n+1}$ denote the hyperbolic space of dimension $n+1$ and let $\partial_{\infty} \bH^{n+1}$ denote its ideal boundary at infinity. The classical asymptotic Plateau problem in hyperbolic space asks whether there exists an area minimizing submanifold $\Sigma \subseteq \bH^{n+1}$ asymptotic to a given submanifold $\Gamma \subseteq \partial_{\infty} \bH^{n+1}$. In a seminal work \cite{Anderson-1}, M.~T.~Anderson solved the problem for absolutely area minimizing varieties for any dimension and codimension; this is one of the most important results in the field and initiates further investigations later on. One motivation is that, if we restrict ourselves to find the desired object in a fixed topological type (say, a disk), then solutions of this problem would yield area minimizing representative of essential surfaces in hyperbolic 3-manifolds. By essential surfaces, we mean $\pi_1$-injective surfaces and they are quite important for understanding the structure of hyperbolic 3-manifolds; see the works of Anderson \cite{Anderson-2} and Gabai \cite{Gabai}, and Uhlenbeck's program \cite{Uhlenbeck} on the moduli spaces of minimal surfaces in hyperbolic $3$-manifolds.

The problem is indeed far-reaching in the sense that several variants of it are also widely applicable in many other fields. For instance, if one looks for entire solutions with prescribed behavior at infinity, then the problem is related to the Van der Waals phase transition model and minimal hypersurfaces; see \cite{PP, Mramor}. On the other hand, from a purely analysis point of view, it is of interest to study Plateau type problems in hyperbolic space, where things change fundamentally due to the geometric nature of $\bH^{n+1}$ having negative sectional curvature; we refer the reader to a work \cite{Lang-2003} of Lang for the problem in Gromov hyperbolic manifolds. For more, the reader should also be referred to the survey \cite{survey} of Coskunuzer.

In this article, we are concerned with a natural generalization of the classical asymptotic Plateau problem in hyperbolic space, which was initiated by Guan and Spruck in \cite{JEMS}. The generalization states the following: 

\begin{displayquote}
Let $f: \bR^n \to \bR$ be a smooth symmetric function of $n$ variables satisfying standard assumptions and let $K$ be an open symmetric convex cone containing the positive cone $K^{+}:=\{\kappa \in \bR^n: \kappa_i>0 \quad \forall\ 1 \leq i \leq n\}$.
Given a disjoint collection of closed embedded smooth $(n-1)$-dimensional submanifolds $\Gamma=\{\Gamma_1,\ldots,\Gamma_m\} \subseteq \partial_{\infty} \bH^{n+1}$ and a constant $0<\sigma<1$, we ask whether there exists a smooth complete hypersurface $\Sigma$ in $\bH^{n+1}$ satisfying 
\begin{equation}
\text{$f(\kappa[\Sigma])=\sigma$ and $\kappa[\Sigma] \in K$ on $\Sigma$}, \label{req1}
\end{equation}
with the asymptotic boundary
\begin{equation}
\partial \Sigma=\Gamma \label{req2}
\end{equation} at infinity; here $\kappa[\Sigma]=(\kappa_1,\ldots,\kappa_n)$ denotes the hyperbolic principal curvatures of $\Sigma$.
\end{displayquote} 

Motivations for considering this generalization were not stated in the work \cite{JEMS} of Guan-Spruck, but from our limited knowledge, besides the ones that are already stated above, it is worth studying for at least two reasons: (1) When restricting to horospherically convex hypersurfaces, solutions to \eqref{req1}-\eqref{req2} can be used to induce complete conformal metrics on subdomains of the sphere $\bS^n$, and thus, it is related to the Yamabe problem and the Min-Oo conjecture; see the correspondence theorem established in \cite{duality} and further developments in \cite{Bonini-Qing-Zhu, Bonini-Ma-Qing, Espinar-CPAM, Espinar-AJM, Espinar-JMPA}. (2) When restricting to locally strictly convex hypersurfaces, it is the dual problem of finding spacelike strictly convex hypersurfaces in the de Sitter space with prescribed future asymptotic boundary, which has some physical applications and we do not address them here but refer the reader to \cite{Spruck-Xiao-1, JDG, Montiel}, and references citing those work.

In their very delicate work \cite{JEMS}, Guan-Spruck proved the existence of $\Sigma$ satisfying \eqref{req1}-\eqref{req2} with the restriction that $\sigma \in (\sigma_0,1)$, where $\sigma_0 \approx 0.37$. Later, Xiao \cite{Xiao-B} improved the bound to $\sigma_0 \approx 0.14$. In a series of joint work \cite{JGA,SGAR,JDG,Xiao-A}, Guan-Spruck-Szapiel-Xiao showed that this existence result could be enhanced to all $\sigma \in (0,1)$ if the cone $K=K^{+}$ coincides with the positive cone. However, the positive cone may not always be the natural cone for consideration. In particular, when $f=H_{k}^{1/k}$ and when $f=(H_k/H_l)^{1/(k-l)}$, where
\[H_{k}(\kappa):=\binom{n}{k}^{-1}S_k(\kappa)=\binom{n}{k}^{-1} \sum_{1 \leq i_1<\cdots<i_k \leq n} \kappa_{i_1}\cdots \kappa_{i_k}\] is the normalized $k$-th elementary symmetric polynomial, it is of desire to have the result for all $\sigma \in (0,1)$ in the $k$-th \Garding\ cone
\[K_k:=\{\kappa \in \bR^n: S_{j}(\kappa)>0 \quad \forall\ 1 \leq j \leq k\}.\] Indeed, these functions do include several notable examples.
\begin{align*}
H_1(\kappa)&=\frac{1}{n}\sum_{i=1}^{n} \kappa_i \quad\quad\quad\quad\quad\quad\quad \text{is the mean curvature},\\
H_2(\kappa)&=\frac{2}{n(n-1)}\sum_{i<j}\kappa_i\kappa_j \quad\quad\quad \text{is the scalar curvature},\\
H_n(\kappa)&=\prod_{i=1}^{n}\kappa_i \quad\quad\quad\quad\quad\quad\quad\quad \text{is the Gauss curvature}, \\
\left(\frac{H_n}{H_{n-1}}\right)(\kappa)&=n\left(\sum_{i=1}^{n} \frac{1}{\kappa_i}\right)^{-1} \quad\quad\quad\ \quad \text{is the harmonic curvature};
\end{align*} other values of $k$ and $l$ also constitute important geometric quantities; see, for example, \cite{Guan-Guan, Guan-Li-Li, Guan-Ma}.

The zero mean curvature case is the one treated in Anderson's seminal work \cite{Anderson-1, Anderson-2}, followed by Hardt-Lin \cite{Hardt-Lin} and Lin \cite{Lin}. It was then extended to the case of constant mean curvature by Guan-Spruck \cite{AJM}, Nelli-Spruck \cite{Nelli-Spruck}, and Tonegawa \cite{Tonegawa}. The Gauss curvature case was raised by Labourie \cite{Labourie} in $\bH^{3}$ for studying the structure of $3$-dimensional hyperbolic manifolds, and was later solved by Rosenberg-Spruck \cite{Rosenberg-Spruck} in $\bH^{n+1}$. The harmonic curvature case is contained in the series of joint work of Guan-Spruck-Szapiel-Xiao \cite{JGA,SGAR,JDG,Xiao-A}. However, the existence result for the scalar curvature case, according to Guan-Spruck \cite{JEMS} and Xiao \cite{Xiao-B}, is only available for a restricted range $\sigma \in (\sigma_0,1)$. Since the scalar curvature case is clearly of great geometric interest, it is in demand to solve \eqref{req1}-\eqref{req2} for all $\sigma \in (0,1)$ when $f=H_{2}^{1/2}$ and $K=K_2$. In this article, we are able to resolve this longstanding problem by introducing new techniques. We state the theorem as follows.
\begin{theorem} \label{k=2}
Given a disjoint collection of closed embedded smooth $(n-1)$-dimensional submanifolds $\Gamma=\{\Gamma_1,\ldots,\Gamma_m\} \subseteq \partial_{\infty} \bH^{n+1}$ and a constant $0<\sigma<1$, if $\Gamma=\partial \Omega$ is mean-convex, then there exists a smooth complete hypersurface $\Sigma$ in $\bH^{n+1}$ satisfying 
\[
\text{$H_{2}^{1/2}(\kappa[\Sigma])=\sigma$ and $\kappa[\Sigma] \in K_2$ on $\Sigma$} 
\]
with the asymptotic boundary
\[
\partial \Sigma=\Gamma 
\] at infinity.
\end{theorem}
In the half-space model, we view $\Gamma=\partial \Omega$ as a codimension one submanifold of the Euclidean space $\bR^n$ with $\Omega$ being a smooth bounded domain. The mean curvature of $\Gamma$ is computed in the Euclidean metric with respect to the interior normal vector. We say $\Gamma$ is mean-convex if its mean curvature is non-negative. Note that the mean-convexity is not an intrinsic hyperbolic notion.
\begin{remark}
The mean-convexity condition on $\partial \Omega$ is imposed to derive global gradient estimates for admissible solutions; see \cite[Proposition 4.1]{JEMS}. Other constraints in the theorem are also necessary. When $\sigma>1$, admissible solutions do not exist by comparing with horospheres; see \cite[Lemma 3.2]{JGA}. Also, when $\Gamma$ is a co-dimension two closed submanifold, there are topological obstructions for $\Gamma$ to bound a hypersurface with $f(\kappa) \in (0,1)$; see \cite{Rosenberg, Rosenberg-Spruck}. The embeddedness assumption on the given asymptotic boundary is essential as well; Lang \cite{Lang-1995} has constructed immersed examples with no solutions exist.
\end{remark}

Our proof follows the framework of Guan-Spruck \cite{JEMS}, in which by seeking $\Sigma=\{(x,u(x))\}$ as the limit of constant curvature graphs (with respect to vertical geodesics)
over a fixed compact domain in a horosphere, they reduced \eqref{req1}-\eqref{req2} to a Dirichlet problem
\begin{equation}
G(D^2u,Du,u)=\sigma, \quad u>0 \quad \text{in $\Omega$}, \label{the PDE}
\end{equation} for a fully nonlinear elliptic partial differential equation with 
\[u=0 \quad \text{on $\partial \Omega$}.\]  The Dirichlet problem is to be solved by the standard continuity method that is invoked along with the Evans-Krylov regularity theorem \cite{Evans, Krylov}, for which one needs to establish a priori estimates for admissible solutions $u$ up to the second order. As a matter of fact, most of these estimates have been perfectly established by Guan-Spruck \cite{JEMS}. In particular, the boundary second derivative estimate
\[\max_{\partial \Omega} |D^2u| \leq C\] is the centerpiece of their paper. The only issue occurs when deriving the global $C^2$ estimate
\begin{equation}
\max_{\Omega} |D^2u| \leq C\left(1+\max_{\partial \Omega} |D^2u|\right), \label{global C2}
\end{equation} and we shall now briefly explain the difficulties. Later, we will present a resolution for achieving Theorem \ref{k=2}.

The first difficulty is that the PDE \eqref{the PDE} is degenerate when $u=0$; see \cite[Section 2]{JEMS}. Alternatively, one may approximate the boundary condition by
\begin{equation}
u=\varepsilon>0 \quad \text{on $\partial \Omega$}. \label{the approximate boundary condition}
\end{equation} Also, for the purpose of taking limits $\varepsilon \to 0$, we would need the constant $C>0$ in \eqref{global C2} be independent of $\varepsilon$; this is the one and only place where Guan-Spruck \cite{JEMS} had to restrict the range of allowable $\sigma$'s.

Our task is thus reduced to obtaining \eqref{global C2} for all possible $\sigma \in (0,1)$. The second difficulty arises from the complicated structure of the scalar curvature equation. Unlike the mean curvature case where the equation is quasi-linear, or the Gauss curvature case where the equation is of Monge-\Ampere\ type, both equations are extensively studied in the literature and many techniques could be adapted to solve this problem. The scalar curvature case does not possess the same advantage. In particular, the recent novel techniques of Guan-Qiu \cite{Guan-Qiu}, Guan-Ren-Wang \cite{Guan-Ren-Wang}, Qiu \cite{Qiu-2}, Sheng-Urbas-Wang \cite{Sheng-Urbas-Wang} and Spruck-Xiao \cite{Spruck-Xiao-2} for solving the scalar curvature equation in the Euclidean space cannot be employed in our setting. Therefore, we will present a new way of obtaining \eqref{global C2} for the scalar curvature equation, which we hope would be inspiring for similar problems; this is one of our major contributions.

Our method consists of several new ingredients and will differ from \cite{SGAR,JEMS, Lu, JDG, Xiao-A, Xiao-B}. The first ingredient is the use of concavity inequalities which amounts to the extraction of enough positive terms from the concavity property of the operator $S_{k}^{1/k}$; this idea comes from the works of Guan-Ren-Wang \cite{Guan-Ren-Wang} and Ren-Wang \cite{Ren-Wang-1, Ren-Wang-2}, and the purpose is to eliminate those troublesome third order terms. For our problem \eqref{req1}-\eqref{req2}, we borrow such an inequality for the $S_2$ operator from a recent work \cite{SY} of Shankar-Yuan, in which they call it the ``almost-Jacobi'' inequality. Applying this inequality drably and following the arguments of Lu \cite{Lu}, we would be able to obtain Theorem \ref{k=2} in dimension $n=4$. Note that, the almost-Jacobi inequality, according to the sharp estimate \cite[Lemma 2.1]{SY}, cannot stay effective for $n \geq 5$. That is to say, in order to obtain Theorem \ref{k=2} in all higher dimensions, we would need more ingredients to the proof and this is the part where we introduce a set of new, original arguments which would reduce the proof to a rather simple situation (see Claim \ref{claim 1}) in which the almost-Jacobi inequality \textit{can} somehow be applicable in all dimensions and this would handle well all those troublesome third order terms. However, this is still not enough because the bad term in our problem is $-C\sum F^{ii}$; see \eqref{the bad term}. Unlike in \cite{Guan-Ren-Wang, Spruck-Xiao-2}, we are not allowed to add the quadratic $|X|^2$ to our auxiliary function. In fact, there seems to be no available techniques in the literature for giving rise to a $+C\sum F^{ii}$ term without resulting in a $C/\varepsilon$ dependence; this situation is severely different from the Euclidean counterparts. The second remarkable feature of our new method is that we \textit{can} deal with this issue (see Claim \ref{claim 2}) by utilizing a particular positive term which has long been neglected by others. In other words, we have successfully found a way to utilize it and consequently, we were able to obtain Theorem \ref{k=2} in all dimensions. For the moment, we shall not continue the elaboration in here, but provide a detailed exposition in section \ref{k=2 proof} to demonstrate our new arguments; see Remark \ref{end remark}.

\begin{remark}
Our new method might be possibly far-reaching, because the applicability of the almost-Jacobi inequality should be restricted due to the dimensional constraint $n \leq 4$ on its effectiveness. However, our proof will demonstrate that, under the settings of some geometric problems, the inequality could play a role in all dimensions, provided that other novel ingredients are invoked along with it. We hope our new arguments would encourage others to use it for possibly more applications because the $S_2$ equation has been studied in many other geometric problems. This can be seen as our secondary contribution.
\end{remark}

Now, we continue the literature review. In \cite{Lu}, by employing a powerful concavity inequality due to Ren-Wang \cite{Ren-Wang-1}, Lu solved \eqref{req1}-\eqref{req2} for all $\sigma \in (0,1)$ when $f=H_{n-1}^{1/(n-1)}$ and $K=K_{n-1}$, $n \geq 3$, which contains the $n=3$ case of our Theorem \ref{k=2}. Then, in \cite{Bin-3}, we followed the method of Lu and obtained the existence result when $f=H_{n-2}^{1/(n-2)}$ and $K=K_{n-2}$, $n \geq 5$, by employing a new concavity inequality of Ren-Wang \cite{Ren-Wang-2}. Incorporating our Theorem \ref{k=2} here, only the intermediate cases $3 \leq k \leq n-3$ are open for $f=H_{k}^{1/k}$ with $K=K_k$. In a recent preprint \cite{Hong-Zhang}, Hong-Zhang was able to obtain \eqref{global C2} for these intermediate cases by assuming additionally that the curvatures $\kappa[\Sigma]=(\kappa_1,\ldots,\kappa_n)$ are uniformly bounded from below by a negative constant. For the quotients $f=(H_k/H_l)^{1/(k-l)}$ in $K=K_k$, we have solved the case when $l=k-1$ in our previous investigation \cite{Bin-1}. For general $0 \leq l \leq k-1$, we would need to impose an additional curvature lower bound to obtain \eqref{global C2}; see \cite{Bin-3}. Some follow-up works on the strictly convex case are given by Hong-Li-Zhang \cite{Tsinghua}, Jiao-Jiao \cite{Jiao-Jiao} and Sui-Sun \cite{Sui-CPAA, Sui-CVPDE, Sui-IMRN}. Uniqueness of solutions are discussed in \cite{Nelli-Zhu, Zeno, SGAR, JDG}. 

To be more comprehensive, we mention the works of Cruz \cite{Cruz}, Guan-Spruck \cite{Guan-Spruck-2002, Guan-Spruck-2004}, Sui \cite{Sui-CPDE}, and Trudinger-Wang \cite{Trudinger-Wang-2002}, for studying \eqref{req1}-\eqref{req2} in $\bR^{n+1}$. On the other hand, for the same asymptotic Plateau problem in $\bH^{n+1}$ but with different classes of fully nonlinear equations, see \cite{He-Tu-Xiang, Chen-Sui-Sun, Yang-Lu}. Finally, we would also like to call attention to a series of articles \cite{Smith-2009, Smith-2013, Smith-preprint} by Smith for some related studies.

The rest of this article is organized as follows. In section \ref{preliminaries}, we review the geometry of graphic hypersurfaces in hyperbolic space and collect a few preliminary lemmas. The crucial curvature estimate will be established in section \ref{k=2 proof} which would yield the desired \eqref{global C2} and Theorem \ref{k=2} would follow as a consequence of the theory in \cite{JEMS}.

\textit{Note.} Some secondary results were included in an early preprint version of this manuscript, which were later removed due to their minor importance; they are now transferred to a separate note \cite{Bin-3}.

\section{Preliminaries} \label{preliminaries}
For the hyperbolic space, we will use the half-space model 
\[\bH^{n+1}=\{(x,x_{n+1}) \in \bR^{n+1}: x_{n+1}>0\}\] equipped with the hyperbolic metric
\[ds^2=\frac{\sum_{i=1}^{n+1} dx_{i}^2}{x_{n+1}^2}.\] In this way, the ideal boundary $\partial_{\infty}\bH^{n+1}$ is naturally identified with $\bR^n=\bR^n \times \{0\}\subseteq \bR^{n+1}$ and \eqref{req2} may be understood in the Euclidean sense. For convenience, we say $\Sigma$ has compact asymptotic boundary if $\partial \Sigma \subseteq \partial_{\infty}\bH^{n+1}$ is compact with respect to the Euclidean metric in $\bR^{n}$.

Let $\Sigma$ be a hypersurface in $\bH^{n+1}$. We shall use $g$ and $\nabla$ to denote the induced hyperbolic metric and the Levi-Civita connection on $\Sigma$, respectively. Viewing $\Sigma$ as a submanifold of $\bR^{n+1}$, we shall use $\tilde{g}$ to denote the induced metric on $\Sigma$ from $\bR^{n+1}$ and $\tilde{\nabla}$ is its Levi-Civita connection.

Throughout, all hypersurfaces in $\bH^{n+1}$ that we consider are assumed to be connected and orientable. If $\Sigma$ is a complete hypersurface in $\bH^{n+1}$ with compact asymptotic boundary at infinity, then the normal vector field of $\Sigma$ is chosen to be the one pointing towards the unique unbounded region in $\bR_{+}^{n+1} \setminus \Sigma$, and both the hyperbolic and Euclidean principal curvatures are calculated with respect to this normal vector field.

Note that
\[\nabla_{\frac{\partial}{\partial x_i}} \frac{\partial}{\partial x_{j}}=\delta_{ij}\frac{1}{x_{n+1}}\frac{\partial}{\partial x_{n+1}} \quad \text{and} \quad \nabla_{\frac{\partial}{\partial x_{\alpha}}} \frac{\partial}{\partial x_{n+1}}=-\frac{1}{x_{n+1}}\frac{\partial}{\partial x_{\alpha}}\] for $ 1 \leq i,j \leq n$ and $1 \leq \alpha \leq n+1$. Suppose $\Sigma$ is the vertical graph of a function $x_{n+1}=u(x_1,\ldots,x_n)$ over a domain $\Omega$ in $\bR^n$. Since the induced tangent vectors on $\Sigma$ are 
\[X_i=\frac{\partial}{\partial x_i}+u_i\frac{\partial}{\partial x_{n+1}},\] the first fundamental form is then given by
\[g_{ij}=\langle X_i,X_j\rangle=\frac{1}{u^2}(\delta_{ij}+u_iu_j)=\frac{\tilde{g}_{ij}}{u^2}\] with the inverse
\[g^{ij}=u^2\left(\delta_{ij}-\frac{u_iu_j}{w^2}\right), \quad \tilde{g}^{ij}=\delta_{ij}-\frac{u_iu_j}{w^2}.\] Hence, with respect to the hyperbolic upward unit normal $\bn=u\nu$, where $\nu$ is the Euclidean normal given by 
\[\nu=\left(\frac{-Du}{w},\frac{1}{w}\right), \quad w=\sqrt{1+|Du|^2}, \quad |Du|^2=\delta^{ij}u_iu_{j}=\sum_{i=1}^{n}u_{i}^2,\]  we use
\[\Gamma_{ij}^{k}=\frac{1}{x_{n+1}}(-\delta_{jk}\delta_{i,n+1}-\delta_{ik}\delta_{j,n+1}+\delta_{ij}\delta_{k,n+1})\] to obtain that
\[\nabla_{X_i}X_j=\left(\frac{\delta_{ij}}{x_{n+1}}+u_{ij}-\frac{u_iu_j}{x_{n+1}}\right)e_{n+1}-\frac{u_je_i+u_ie_j}{x_{n+1}}.\] Thus,
\begin{align*}
h_{ij}&=\langle \nabla_{X_i}X_j,\bn\rangle=\frac{1}{uw}\left(\frac{\delta_{ij}}{u}+u_{ij}-\frac{u_iu_j}{u}+2\frac{u_iu_j}{u}\right)\\
&=\frac{1}{u^2w}\left(\delta_{ij}+u_iu_j+uu_{ij}\right)=\frac{\tilde{h}_{ij}}{u}+\frac{\nu^{n+1}}{u^2}\tilde{g}_{ij}.
\end{align*} The hyperbolic principal curvatures $\kappa_i$ of $\Sigma$ are the roots of the characteristic equation
\[\det(h_{ij}-\kappa g_{ij})=u^{-n}\det \left(\tilde{h}_{ij}-\frac{1}{u}\left(\kappa-\frac{1}{w}\right)\tilde{g}_{ij}\right)=0.\] In particular, we have the following simple relation:
\begin{equation}
\kappa_i=x_{n+1}\tilde{\kappa}_{i}+\nu^{n+1}, \quad 1 \leq i \leq n. \label{simple relation}
\end{equation}

\begin{remark}
The component 
\[\nu^{n+1}=\frac{1}{\sqrt{1+|Du|^2}}\] will prove to be useful in section \ref{k=2 proof}.
\end{remark}

Let $A=\{a_{ij}\}$ be an $n \times n$ symmetric matrix and let 
\[F: \{\text{$n \times n$ symmetric matrices}\} \to \bR\] be an operator of the form $F(A)=f(\lambda(A))$, where $\lambda(A)=(\lambda_1,\ldots,\lambda_n)$ are the eigenvalues of $A$ and $f: \bR^n \to \bR$ is some smooth symmetric function of $n$ variables. We denote by
\[F^{ij}(A):=\frac{\partial F}{\partial a_{ij}}(A), \quad F^{ij,rs}(A)=\frac{\partial^2 F}{\partial a_{ij} \partial a_{rs}}(A).\] Since $A$ is symmetric, the matrix $F^{ij}(A)$ is symmetric as well. Moreover, when $A$ is diagonal, we have $F^{ij}=f_i\delta_{ij}$, where
\[f_i:=\frac{\partial f}{\partial \lambda_i}.\] The equation \eqref{req1} can be written in a local frame as 
\[F(A[\Sigma])=\sigma\] for $A[\Sigma]=\{g^{ik}h_{kj}\}$. 

Next, we observe that
\[\frac{\partial}{\partial \kappa_i} S_k(\kappa_1,\ldots,\kappa_n)=S_{k-1}(\kappa_1,\ldots,\kappa_{i-1},0,\kappa_{i+1},\ldots,\kappa_n).\] We will denote the partial derivatives by the following notations
\[\frac{\partial}{\partial \kappa_i}S_k (\kappa_1,\ldots,\kappa_n)=S_{k-1}(\kappa|i)=S_{k}^{ii}(\kappa)\]
interchangeably. The notations $S_{k}^{ii,jj}(\kappa)$ and $S_{k-2}(\kappa|ij)$ are defined in a similar way for second order partial derivatives. Standard properties of the elementary symmetric polynomials can be found in \cite{Wang-book, Spruck-MSRI, Lieberman}. Here we prove two lemmas that are specialized for the $S_2$ operator.

\begin{lemma}\label{sharp 1}
Let $n \geq 2$.
Suppose that $\kappa=(\kappa_1,\ldots,\kappa_n) \in K_2$ is ordered as $\kappa_1 \geq \kappa_2 \geq \cdots \geq \kappa_n$. Then we have
\[\kappa_n > -\frac{n-2}{n}\cdot S_1(\kappa).\]
\end{lemma}

\begin{proof}[Proof of Lemma \ref{sharp 1}]
The lemma is trivial if $n=2$ or $\kappa_n \geq 0$. Assume $n \geq 3$ and $\kappa_n<0$, we write $\kappa'=(\kappa|n)=(\kappa_1,\ldots,\kappa_{n-1})$. From the identity
\[S_2(\kappa)=\kappa_nS_1(\kappa|n)+S_2(\kappa|n),\] it follows that
\[\frac{S_1(\kappa)}{-\kappa_n}=\frac{[S_1(\kappa)-\kappa_n]+\kappa_n}{-\kappa_n}=\frac{S_1(\kappa')+\kappa_n}{-\kappa_n}=-1+\frac{[S_1(\kappa')]^2}{S_2(\kappa')-S_2(\kappa)}.\] Note that since $\kappa_n<0$ and $\kappa \in K_2$, we have $0<S_2(\kappa')-S_2(\kappa)<S_2(\kappa')$. Hence,
\[\frac{S_1(\kappa)}{-\kappa_n}>-1+\frac{[S_1(\kappa')]^2}{S_2(\kappa')}.\]
The lemma will follow by expressing $S_1(\kappa')^2$ as a scalar multiple of $S_2(\kappa')$.

The quantity $S_1(\kappa')$ can be seen as the trace of some $(n-1) \times (n-1)$ matrix $A$ with eigenvalues $\kappa'=(\kappa_1,\ldots,\kappa_{n-1})$. The traceless part is then $A-\frac{\tr(A)I}{n-1}$ whose eigenvalues are $\kappa_{i}'-\frac{S_1(\kappa')}{n-1}$; we will denote this vector by $\kappa'^{\perp}$. Now, by a direct computation, we have
\begin{align*}
|\kappa'^{\perp}|^2&=\sum_{i=1}^{n-1}\left[\kappa_{i}^{'}-\frac{S_1(\kappa')}{n-1}\right]^2 \\
&=\sum_{i=1}^{n-1} \left[\kappa_{i}'^{2}-\frac{2S_{1}(\kappa')}{n-1}\kappa_i'+\frac{S_{1}(\kappa')^2}{(n-1)^2}\right] \\
&=\left(\sum_{i=1}^{n-1} \kappa_{i}'^{2}\right)-\frac{2[S_{1}(\kappa')]^2}{n-1}+\frac{[S_{1}(\kappa')]^2}{n-1}\\
&=\left([S_{1}(\kappa')]^2-2S_2(\kappa')\right)-\frac{2[S_{1}(\kappa')]^2}{n-1}+\frac{[S_{1}(\kappa')]^2}{n-1}\\
&=\frac{n-2}{n-1}S_{1}(\kappa')^2-2S_2(\kappa'),
\end{align*} where we have used the elementary identity
\begin{equation}
2S_2(\kappa)=S_1(\kappa)^2-\sum_{i=1}^{n}\kappa_{i}^2, \quad \kappa \in \bR^n. \label{elementary identity}
\end{equation} 

\begin{remark}
We would like to thank Ravi Shankar for helping us understand this step.
\end{remark}

Rearranging, we have
\[[S_{1}(\kappa')]^2=\frac{n-1}{n-2}[2S_2(\kappa')+|\kappa'^{\perp}|^2] \geq 2\frac{n-1}{n-2}S_2(\kappa')\] and the desired bound 
\[\frac{S_1(\kappa)}{-\kappa_n}>-1+\frac{[S_1(\kappa')]^2}{S_2(\kappa')} \geq -1+\frac{2(n-1)}{n-2}=\frac{n}{n-2}\] follows.
\end{proof}

\begin{lemma} \label{sharp 2}
Let $\kappa=(\kappa_1,\ldots,\kappa_n) \in K_2$ be ordered as $\kappa_1 \geq \kappa_2 \geq \cdots \geq \kappa_n$ and let $f(\kappa)=S_2(\kappa)$. Then
\[\frac{S_2(\kappa)}{S_1(\kappa)} \leq f_1(\kappa) \leq \left(\frac{n-1}{n}\right)S_1(\kappa)\] and 
\[\left(1-\frac{1}{\sqrt{2}}\right)S_1(\kappa) \leq f_i(\kappa) \leq 2\left(\frac{n-1}{n}\right)S_1(\kappa), \quad i \geq 2.\]
\end{lemma}

\begin{proof}[Proof of Lemma \ref{sharp 2}]
Observe that for all $1 \leq i \leq n$, we have $f_i(\kappa)=S_1(\kappa)-\kappa_i$. The upper bound for $f_{1}$ then comes from the fact that $S_1=\kappa_1 + \cdots + \kappa_n \leq n \kappa_1$. On the other hand, the upper bound for $i \geq 2$ follows from Lemma \ref{sharp 1}:
\[f_i \leq f_n=S_1-\kappa_n<\left(1+\frac{n-2}{n}\right)S_1.\]

For the lower bounds, we use the elementary identity \eqref{elementary identity} to get 
\[f_1=S_1-\kappa_1=\frac{2S_2+\sum_{j \neq 1}\kappa_{j}^2}{S_1+\kappa_1}\geq \frac{2S_2}{2S_1}=\frac{S_2}{S_1}.\] Similarly, when $i \geq 2$, we have the bound already if $\kappa_i \leq 0$; while if $\kappa_i>0$, we have
\[\kappa_i \leq \sqrt{\frac{\kappa_{1}^2+\cdots+\kappa_{i}^2}{i}}\] by the AM-GM inequality
\[\frac{\kappa_{1}^2+\cdots+\kappa_{i}^2}{i} \geq (\kappa_{1}^2\cdots \kappa_{i}^2)^{\frac{1}{i}} \geq \kappa_{i}^2.\]
Thus, we have
\[f_i=S_1-\kappa_i \geq S_1-\sqrt{\frac{\kappa_{1}^2+\cdots+\kappa_{i}^2}{i}} \geq (1-i^{-1/2})S_1\] where we have again used the elementary identity \eqref{elementary identity} i.e.
\[S_1=\sqrt{2S_2+|\kappa|^2} \geq \sqrt{\kappa_{1}^2+\cdots+\kappa_{i}^2}.\]
\end{proof}
In \cite{SY}, Shankar-Yuan proved a so-called ``almost-Jacobi'' inequality for $2$-convex solutions to the quadratic Hessian equation
\begin{equation}
S_{2}(\lambda(D^2u))=\const, \quad \lambda(D^2u) \in K_2, \quad \text{in $B_1$}. \label{quadratic Hessian equation with a constant right-hand side}
\end{equation} By following the exact same procedure in \cite{SY}, it can be readily verified that the inequality also holds for $2$-convex hypersurfaces satisfying the constant scalar curvature equation in $\bH^{n+1}$ because the only change occurs for the fourth order terms where we would have
\[h_{iikk}=h_{kkii}+(\kappa_i\kappa_k-1)(\kappa_i-\kappa_k)\] in place of
\[u_{iikk}=u_{kkii}.\] This would only result in some additional curvature terms of no harm; a redundant verification can be found in \cite{Bin-thesis}. 

\begin{proposition} \label{almost-Jacobi}
Let $\Sigma$ be a smooth complete hypersurface in $\bH^{n+1}$ which solves
\[S_2(\kappa[\Sigma])=\const, \quad \kappa[\Sigma] \in K_2 \quad \text{on $\Sigma$}.\] Then at a point $X_0 \in \Sigma$ where the second fundamental form $h_{ij}=\kappa_i\delta_{ij}$ is diagonal and $\kappa_1 \geq \cdots \geq \kappa_n$, we have, for the quantity $b=\log S_1(\kappa)$, that
\begin{gather}
\begin{split}
\sum_{i=1}^{n} F^{ii}\nabla_{ii} b \geq &\ \alpha_n\left(\beta_n+\frac{\kappa_{n}}{S_1(\kappa)}\right)\sum_{i=1}^{n} F^{ii}(\nabla_{i}b)^2-\left(\sum_{i=1}^{n} F^{ii} + \sum_{i=1}^{n} F^{ii}\kappa_{i}^2\right) \\
&+2F\cdot S_1(\kappa)-\frac{4F^2}{S_1(\kappa)},
\end{split} \label{almost-Jacobi inequality}
\end{gather} where
\[\alpha_n:=\frac{\sqrt{3n^2+1}-(n+1)}{3(n-1)}, \quad \beta_n:=\frac{\sqrt{3n^2+1}-(n-1)}{2n}.\]
\end{proposition}

\begin{remark}
This inequality could work equally well in our proof of the curvature estimate if we instead tried to derive it for the quantity $b=\log \kappa_{\max}$, because our use of \eqref{almost-Jacobi inequality} is only to control those troublesome third order terms. Following the terminology of Yuan \cite{Yuan-QHE}, we call it the ``maximum eigenvalue Jacobi inequality'', and the one for $b=\log S_1(\kappa)$ is called the ``trace Jacobi inequality''. Besides the convenience of citing directly from Shankar-Yuan's paper \cite{SY}, the trace Jacobi inequality is usually preferred in at least two aspects. The first advantage is that one would not need to worry about the issue of $\kappa_{\max}$ being non-smooth; see \cite[Proposition 2.4]{Warren-Yuan}. Secondly, according to Shankar-Yuan \cite{SY-Duke}, the trace Jacobi inequality could ``rescue'' the saddleness of semi-convex solutions after the Legendre-Lewy transformation and would help yield rigidity results for such solutions. However, as a compensation, the computational efforts for deriving a trace Jacobi inequality would be much more involved. In fact, it seems to us that the trace Jacobi inequality was not anticipated to hold true until a proof was presented by Qiu \cite{Qiu-1, Qiu-2} for $2$-convex solutions in dimension three. Later, Shankar-Yuan \cite{SY-Duke} was able to obtain the trace Jacobi inequality for semi-convex solutions to \eqref{quadratic Hessian equation with a constant right-hand side} in all dimensions. When $n \geq 4$, the Jacobi inequality would fail without stronger convexity conditions. In the most recent work \cite{SY}, Shankar-Yuan obtained a weaker form for $b=\log \Delta u$ in dimension four with $u$ being $2$-convex, which they call it ``the almost-Jacobi inequality''.
\end{remark}

\begin{remark}
The longstanding problem of obtaining purely interior $C^2$ estimates for $2$-convex solutions to the quadratic Hessian equation and the prescribed scalar curvature equation has attracted much attention in the past two decades. The Jacobi inequality has always been one of the major ingredients for reaching the desired a priori estimates. Due to our unfamiliarity with the relevant theories, we do not comment further but refer the interested reader to the work of Guan-Qiu \cite{Guan-Qiu}, Lu \cite{Lu-1, Lu-2}, Qiu \cite{Qiu-1, Qiu-2}, Shankar-Yuan \cite{SY-CVPDE, SY}, and Warren-Yuan \cite{Warren-Yuan}; see also Yuan's survey \cite{Yuan-QHE} and the compactness argument of McGonagle-Song-Yuan \cite{MSY}.
\end{remark}

\begin{remark}
Despite the important role that the Jacobi inequality has played in deriving the purely interior $C^2$ estimates for the $S_2$ equations, there are two papers that entail the Jacobi inequalities but had not receive enough attention. The first one is Chen's investigation \cite{Chen} on optimal concavity of the $S_2$ operator. Guan-Qiu \cite{Guan-Qiu} has demonstrated how to use this optimal concavity result to derive the maximum eigenvalue Jacobi inequality in all dimensions for solutions whose $S_3$'s are bounded from below. The second paper is due to Ren-Wang \cite{Ren-Wang-1}, in which they have proved a powerful concavity inequality for the $S_{n-1}$ operator. Taking $n=3$ will lead to the maximum eigenvalue Jacobi inequality in dimension three, which was obtained by Warren-Yuan \cite{Warren-Yuan} in 2009 for a constant right-hand side. The derivation using Ren-Wang's concavity inequality remains valid even when the right-hand side of the equation is non-constant; see also \cite{Lu-Tsai} and \cite[Lemma 1.4]{Tu}.
\end{remark}

\section{The curvature estimate} \label{k=2 proof}
In this section, we prove Theorem \ref{k=2}. According to Guan and Spruck, as they have commented in \cite{JEMS}, it suffices to derive a global curvature estimate for admissible graphs which would yield the desired \eqref{global C2}; everything else has been perfectly established in their paper \cite{JEMS}.

\begin{theorem}
Suppose $\Gamma=\partial \Omega$ is mean-convex and $\sigma \in (0,1)$. If $\Sigma=\{(x,u(x))\}$ is a $2$-convex graph with $u$ being a smooth solution of \eqref{the PDE} and \eqref{the approximate boundary condition}, then there exists some $C>0$ depending only on $n$, $\sigma$, and not on $\varepsilon$, such that
\[\max_{\Sigma} |H(\kappa)| \leq C \left(1+ \max_{\partial \Sigma}|H(\kappa)|\right),\] where $H(\kappa):=H_1(\kappa)=\frac{1}{n}\sum_{i=1}^{n} \kappa_i$ is the mean curvature of $\Sigma$.
\end{theorem}

\begin{proof}
Since $K_{2} \subseteq K_1=\{H>0\}$, it suffices to establish an upper bound. Also, in order to invoke the almost-Jacobi inequality \eqref{almost-Jacobi inequality} more conveniently, we will work with the equation
\[S_{2}(\kappa[\Sigma])=\binom{n}{2}\sigma^2.\] Consider the following test function:
\[Q= b - N \log \nu^{n+1},\] where $b=\log S_1$ and $N>0$ is a constant to be determined later. 

\begin{remark}
The mean-convexity assumption on $\Gamma$ will yield $\inf \nu^{n+1} \geq \sigma>0$; see \cite[Proposition 4.1]{JEMS}.
\end{remark}

\begin{remark}
A remarkable feature is that our choice of the parameter $N$ will be different than the ones in the literature. In \cite{Lu}, Lu assumed a large value for $N$; while in \cite{JEMS, SGAR}, Guan and Spruck used the term $\log (\nu^{n+1}-a)$ with $0<2a\leq \nu^{n+1}$ being small. Here, we are going to choose an intermediate value $N=3+2\sqrt{3} \approx 6.464$, which is neither large nor small, but it will gauge everything well. 
\end{remark}

Suppose that $Q$ attains its maximum at some interior $X_0 \in \Sigma$. Let $\{\tau_1,\ldots,\tau_n\}$ be a local orthonormal frame around $X_0$ such that the second fundamental form $h_{ij}=\kappa_i\delta_{ij}$ is diagonalized and the principal curvatures are ordered as
\[\kappa_1 \geq \kappa_2 \geq \cdots \geq \kappa_n.\] In what follows, we will carry out all calculations at the point $X_0$ without explicitly indicating so in every step. For notational convenience, we shall write $v_{i}=\nabla_iv$ and $v_{ij}=\nabla_{ij}v$ for a smooth function $v$ and write $h_{ijk}=\nabla_kh_{ij}$, $h_{ijkl}=\nabla_{lk}h_{ij}$, etc.
Thus, at $X_0$, we have
\begin{align}
0&=b_i - N\frac{\nabla_i\nu^{n+1}}{\nu^{n+1}}, \label{k=2 1st critical}\\
0&\geq b_{ii} - \frac{N}{\nu^{n+1}}\nabla_{ii}\nu^{n+1} + N \frac{(\nabla_{i}\nu^{n+1})^2}{(\nu^{n+1})^2}. \label{k=2 2nd critical 1}
\end{align}
Contracting \eqref{k=2 2nd critical 1} with $F=S_2$, we have
\begin{equation}
0 \geq \sum_{i=1}^{n} F^{ii}b_{ii} - \frac{N}{\nu^{n+1}}\sum_{i=1}^{n} F^{ii}\nabla_{ii}\nu^{n+1} +N \sum_{i=1}^{n} F^{ii}\frac{(\nabla_i\nu^{n+1})^2}{(\nu^{n+1})^2}. \label{k=2 2nd critical 2}
\end{equation}
\begin{remark} \label{reason 1}
Note that, here we have kept the term 
\[N\frac{(\nabla_i\nu^{n+1})^2}{(\nu^{n+1})^2},\] which was plausibly omitted in Lu's proof \cite{Lu}, because it would hardly be of any help. We will soon demonstrate how to utilize its presence below.
\end{remark}

We can calculate \cite[Lemma 4.3]{SGAR}
\begin{align*}
&\ \sum_{i=1}^{n} F^{ii}\nabla_{ii}\nu^{n+1} \\
=&\ 2\sum_{i=1}^{n} F^{ii}\frac{u_i}{u}\nabla_{i} \nu^{n+1}+2F[1+(\nu^{n+1})^2]-\nu^{n+1}\left(\sum_{i=1}^{n} F^{ii} + \sum_{i=1}^{n} F^{ii} \kappa_{i}^2\right).
\end{align*} With this and the almost-Jacobi inequality \eqref{almost-Jacobi inequality}, the second order critical condition \eqref{k=2 2nd critical 2} becomes
\begin{gather} \label{k=2 2nd critical 3}
\begin{split}
0 & \geq \sum_{i=1}^{n} F^{ii}b_{ii} - \frac{N}{\nu^{n+1}}\sum_{i=1}^{n} F^{ii}\nabla_{ii}\nu^{n+1} +N \sum_{i=1}^{n} F^{ii}\frac{(\nabla_i\nu^{n+1})^2}{(\nu^{n+1})^2}\\
&\geq  \left[\alpha_n\left(\beta_n+\frac{\kappa_n}{S_1}\right)\sum_{i=1}^{n} F^{ii}b_{i}^2 -\left(\sum_{i=1}^{n} F^{ii}+\sum_{i=1}^{n} F^{ii}\kappa_{i}^2\right)+2FS_1\right]\\
&\quad - \frac{N}{\nu^{n+1}}\left[2\sum_{i=1}^{n} F^{ii}\frac{u_i}{u}\nabla_i\nu^{n+1}-\nu^{n+1}\left(\sum_{i=1}^{n} F^{ii}+\sum_{i=1}^{n} F^{ii}\kappa_{i}^2\right)\right]\\
&\quad +N \sum_{i=1}^{n} F^{ii}\frac{(\nabla_i\nu^{n+1})^2}{(\nu^{n+1})^2}-2NF\frac{[1+(\nu^{n+1})^2]}{\nu^{n+1}} - \frac{C(n,\sigma)}{S_1}\\
&\geq  \alpha_n \left(\beta_n+\frac{\kappa_n}{S_1}\right) \sum_{i=1}^{n} F^{ii}b_{i}^2 + (N-1)\left(\sum_{i=1}^{n} F^{ii} + \sum_{i=1}^{n} F^{ii}\kappa_{i}^2\right)\\
&\quad + N \sum_{i=1}^{n} F^{ii}\frac{(\nabla_i\nu^{n+1})^2}{(\nu^{n+1})^2}-2N\sum_{i=1}^{n} F^{ii}\frac{u_i}{u}\frac{\nabla_i\nu^{n+1}}{\nu^{n+1}} +FS_{1}-CN
\end{split}
\end{gather} for some $C>0$ depending on $n$ and $\sigma$, where we have also assumed that $S_1$ is sufficiently large so that
\[FS_1 \geq \frac{C(n,\sigma)}{S_1}\] to get the third inequality. 

Now, one major difficulty is that, due to a sharp estimate of Shankar-Yuan \cite[Lemma 2.1]{SY}, the coefficient
\[\alpha_n\left(\beta_n+\frac{\kappa_n}{S_1}\right)\] cannot stay positive when $n \geq 5$. Note that the situation would have become much worse if that coefficient was indeed negative. In order to resolve this issue, we utilize the extra term that is mentioned in Remark \ref{reason 1} along with some novel arguments and claim the following.

\begin{claim} \label{claim 1}
By choosing an appropriate value for $N$, we can have that
\[\kappa_n (X_0)\geq -C\] for some $C>0$ depending on $n, N$ and $\sigma$.
\end{claim}
\begin{proof}[Proof of Claim \ref{claim 1}]
Note that if $\kappa_n \geq 0$, then the claim is trivial. Assume $\kappa_n <0$, by using the first order critical equation \eqref{k=2 1st critical} and Lemma \ref{sharp 1}, we have that
\begin{align*}
&\ \alpha_n\left(\beta_n+\frac{\kappa_n}{S_1}\right)\sum_{i=1}^{n} F^{ii}b_{i}^2 +N \sum_{i=1}^{n} F^{ii}\frac{(\nabla_i\nu^{n+1})^2}{(\nu^{n+1})^2}\\
=&\ \alpha_n\left(\beta_n+\frac{\kappa_n}{S_1}\right)N^2\sum_{i=1}^{n} F^{ii}\frac{(\nabla_{i}\nu^{n+1})^2}{(\nu^{n+1})^2}+N\sum_{i=1}^{n} F^{ii}\frac{(\nabla_i \nu^{n+1})^2}{(\nu^{n+1})^2}\\
\geq &\ \alpha_n\left(\beta_n-\frac{n-2}{n}\right)N^2\sum_{i=1}^{n} F^{ii}\frac{(\nabla_{i}\nu^{n+1})^2}{(\nu^{n+1})^2}+N\sum_{i=1}^{n} F^{ii}\frac{(\nabla_i \nu^{n+1})^2}{(\nu^{n+1})^2}\\
=&\ \left[\alpha_n\left(\beta_n-\frac{n-2}{n}\right)N+1\right]N\sum_{i=1}^{n} F^{ii}\frac{(\nabla_i\nu^{n+1})^2}{(\nu^{n+1})^2}
\end{align*} Now, recall the values
\[\alpha_n=\frac{\sqrt{3n^2+1}-(n+1)}{3(n-1)}, \quad \beta_n=\frac{\sqrt{3n^2+1}-(n-1)}{2n},\] it is elementary to verify that
\[\alpha_n\left(\beta_n-\frac{n-2}{n}\right) \geq -a_0\] for all $n>1$, where $a_0:=\frac{2}{\sqrt{3}}-1$. By choosing $N \leq \frac{1}{a_0}=3+2\sqrt{3}$, we have that
\[\alpha_n\left(\beta_n-\frac{n-2}{n}\right)N+1 \geq -a_0N+1 \geq 0\] and so the sum of these two terms is non-negative. 

\begin{remark}
If we had not kept the term mentioned in Remark \ref{reason 1}, we would not have obtained the non-negativity here.
\end{remark}

Next, using \cite[Lemma 4.1]{Sui-CVPDE},
\begin{equation}
\nabla_i\nu^{n+1}=\frac{u_i}{u}(\nu^{n+1}-\kappa_i), \quad \frac{u_{i}^2}{u^2} \leq \sum_{i=1}^{n} \frac{u_{i}^2}{u^2} \leq 1,\label{geometric formulas for APP}
\end{equation} we can perform a reduction technique that was exhibited in our previous work \cite{Bin-1}; see also \cite{Bin-3}. The key is to observe the utility of $\sum F^{ii}\kappa_{i}^2$, and we would have
\begin{align*}
&\ \frac{N-1}{2}\sum_{i=1}^{n} F^{ii}\kappa_{i}^2 - 2N \sum_{i=1}^{n} F^{ii}\frac{u_{i}}{u}\frac{\nabla_i \nu^{n+1}}{\nu^{n+1}} \\
=&\ \frac{N-1}{2}\sum_{i=1}^{n} F^{ii}\kappa_{i}^2 + 2N \sum_{i=1}^{n} F^{ii}\frac{u_{i}^2}{u^2}\frac{\kappa_i-\nu^{n+1}}{\nu^{n+1}}\\
\geq &\ \frac{N-1}{2}\sum_{i=1}^{n} F^{ii}\kappa_{i}^2 + 2N \sum_{\kappa_i<\nu^{n+1}} F^{ii}\frac{\kappa_i-\nu^{n+1}}{\nu^{n+1}}\\
\geq &\ \sum_{\kappa_i<\nu^{n+1}} F^{ii}\left[\frac{N-1}{2}\kappa_{i}^2+\frac{2N}{\nu^{n+1}}\kappa_i-2N\right]. \\
\end{align*} Since the quadratic polynomial
\[\frac{N-1}{2}\kappa_{i}^2+\frac{2N}{\nu^{n+1}}\kappa_i-2N\] is non-negative for 
\[\kappa_i \leq -\frac{\frac{2N}{\nu^{n+1}}+2\sqrt{\frac{1+(\nu^{n+1})^2}{(\nu^{n+1})^2}N^2-N}}{N-1}.\] For simplicity and the fact \cite[Proposition 4.1]{JEMS} that $\sigma \leq \nu^{n+1} \leq 1$, we may take 
\[\eta:=\frac{N}{N-1}\frac{2}{\sigma}(1+\sqrt{2})>\frac{\frac{2N}{\nu^{n+1}}+2\sqrt{\frac{1+(\nu^{n+1})^2}{(\nu^{n+1})^2}N^2-N}}{N-1}.\] From \eqref{k=2 2nd critical 3} and Lemma \ref{sharp 2}, we are then left with
\begin{gather} \label{trick 1}
\begin{split}
0 & \geq \frac{N-1}{2}\sum_{i=1}^{n} F^{ii}\kappa_{i}^2 + 2N\sum_{-\eta<\kappa_i<\nu^{n+1}} F^{ii}\frac{u_{i}^2}{u^2}\frac{\kappa_i-\nu^{n+1}}{\nu^{n+1}} -CN\\
&\geq \frac{N-1}{2} F^{nn} \kappa_{n}^2 - 2N \frac{\eta+1}{\sigma} \sum_{i=1}^{n} F^{ii} \\
&\geq C(n,N) \sum_{i=1}^{n} F^{ii} \cdot \kappa_{n}^2 - C(N,\eta,\sigma) \cdot \sum_{i=1}^{n} F^{ii}.
\end{split}
\end{gather} Finally, we divide \eqref{trick 1} by $\sum F^{ii}$ and rearranging, we obtain that
\begin{equation}
\kappa_{n}^2 \leq C \label{trick 2}
\end{equation} for some $C>0$ depending on $n,\sigma$ and $N$.

The claim is proved.
\begin{remark}
We learned of the dividing-$\sum F^{ii}$-and-rearranging trick from \cite{Chu-Jiao}. 
\end{remark}
\end{proof}

Now, back to \eqref{k=2 2nd critical 3} which we restate it here for convenience.
\begin{gather}\label{k=2 2nd critical 4}
\begin{split}
0&\geq  \alpha_n \left(\beta_n+\frac{\kappa_n}{S_1}\right) \sum_{i=1}^{n} F^{ii}b_{i}^2 + (N-1)\left(\sum_{i=1}^{n} F^{ii} + \sum_{i=1}^{n} F^{ii}\kappa_{i}^2\right)\\
&\quad + N \sum_{i=1}^{n} F^{ii}\frac{(\nabla_i\nu^{n+1})^2}{(\nu^{n+1})^2}-2N\sum_{i=1}^{n} F^{ii}\frac{u_i}{u}\frac{\nabla_i\nu^{n+1}}{\nu^{n+1}} +FS_{1}-CN.
\end{split}
\end{gather}
By assuming $S_1$ is sufficiently large and invoking Claim \ref{claim 1}, we have
\[\beta_n+\frac{\kappa_n}{S_1} \geq \beta_n-\frac{C}{S_1} \geq (1-\theta)\beta_n\] for some $0<\theta<1$ to be determined. 
\begin{claim} \label{claim 2}
For $N= \frac{1}{a_0}=3+2\sqrt{3}$, we have
\begin{align*}
&\ (N-1)\sum_{i=1}^{n} F^{ii}+\alpha_n \left(\beta_n+\frac{\kappa_n}{S_1}\right) \sum_{i=1}^{n} F^{ii}b_{i}^2 \\
+ &\ N \sum_{i=1}^{n} F^{ii}\frac{(\nabla_i\nu^{n+1})^2}{(\nu^{n+1})^2}-2N\sum_{i=1}^{n} F^{ii}\frac{u_i}{u}\frac{\nabla_i\nu^{n+1}}{\nu^{n+1}}\\
\geq 0.
\end{align*}
\end{claim}
\begin{proof}[Proof of Claim \ref{claim 2}]
Using the first order critical condition \eqref{k=2 1st critical} and the formulas \eqref{geometric formulas for APP}, we first derive
\begin{align*}
&\ \alpha_n \left(\beta_n+\frac{\kappa_n}{S_1}\right) \sum_{i=1}^{n} F^{ii}b_{i}^2+N \sum_{i=1}^{n} F^{ii} \frac{(\nabla_i\nu^{n+1})^2}{(\nu^{n+1})^2}-2N\sum_{i=1}^{n} F^{ii}\frac{u_i}{u}\frac{\nabla_i\nu^{n+1}}{\nu^{n+1}}\\
\geq &\ (1-\theta)\alpha_n\beta_n N^2\sum_{i=1}^{n} F^{ii}\frac{(\nabla_i \nu^{n+1})^2}{(\nu^{n+1})^2} +N \sum_{i=1}^{n} F^{ii} \frac{(\nabla_i\nu^{n+1})^2}{(\nu^{n+1})^2}\\
&\quad -2N\sum_{i=1}^{n} F^{ii}\frac{u_i}{u}\frac{\nabla_i\nu^{n+1}}{\nu^{n+1}}\\
=&\ \left[(1-\theta)\alpha_n\beta_n N^2+N\right]\sum_{i=1}^{n} F^{ii}\frac{u_{i}^2}{u^2}\left(\frac{\kappa_i-\nu^{n+1}}{\nu^{n+1}}\right)^2  +2N\sum_{i=1}^{n} F^{ii}\frac{u_{i}^2}{u^2}\frac{\kappa_i-\nu^{n+1}}{\nu^{n+1}}
\end{align*} whose summand is a quadratic polynomial in 
\[t=\frac{\kappa_i-\nu^{n+1}}{\nu^{n+1}}.\]
Note that, when $A>0$, in order to have $At^2+Bt>-c_0$, we need the discriminant be strictly negative, that is,
\[c_0>\frac{B^2}{4A}.\] In our case, we have that
\[\frac{B^2}{4A}=\frac{4N^2}{4[(1-\theta)\alpha_n\beta_nN^2+N]}=\frac{1}{(1-\theta)\alpha_n\beta_n+a_0}.\]

Now, we want that
\begin{equation}
(N-1)\sum_{i=1}^{n} F^{ii} -c_0 \sum_{i=1}^{n} F^{ii} \geq 0. \label{the bad term}
\end{equation} In other words, the constant $c_0$ must satisfy
\[N-1 \geq c_0 > \frac{1}{(1-\theta)\alpha_n\beta_n+a_0}.\] For this to happen, we just need to ensure that
\[N-1>\frac{1}{(1-\theta)\alpha_n\beta_n+a_0}.\] For simplicity, take some $\varepsilon>0$ to be determined later, and for this $\varepsilon>0$, there exists some $\theta>0$ such that
\[\frac{1}{(1-\theta)\alpha_n\beta_n+a_0}<\frac{1}{\alpha_n\beta_n+a_0}+\varepsilon.\] Furthermore, for notational convenience, we take 
\[\varepsilon=\frac{\delta}{\alpha_n\beta_n+a_0}\] for another $\delta>0$ to be determined later, so that the task is reduced to proving
\[N-1>\frac{1+\delta}{\alpha_n\beta_n+a_0},\] or equivalently,
\[\alpha_n\beta_n>\frac{1+\delta}{N-1}-a_0=\frac{1+\delta}{N-1}-\frac{1}{N}=\frac{\delta N+1}{N(N-1)}\] for some $\delta>0$. Indeed, one choice would be to take $\delta=\frac{1}{N}=a_0$ so that $\delta N =1$, and it can be readily verified that the following strict inequality
\[\alpha_n\beta_n>\frac{2}{N(N-1)} \approx 0.0566\]
holds for all $n>1$ and the claim is proved.

\begin{remark}
The appearance of $a_0$ in the denominator is crucial; this justifies again the necessity of keeping the term
\[N \frac{(\nabla_i\nu^{n+1})^2}{(\nu^{n+1})^2}.\]
\end{remark}

\end{proof}
Hence, by \eqref{k=2 2nd critical 4}, Claim \ref{claim 1} and Claim \ref{claim 2}, we are left with
\[0 \geq FS_1 - CN\] and the desired estimate follows.

The proof is now complete.
\end{proof}

\begin{remark} \label{end remark}
The act of keeping the term
\[N\frac{(\nabla_i\nu^{n+1})^2}{(\nu^{n+1})^2},\] and establishing
Claim \ref{claim 1} and Claim \ref{claim 2} constitute the most crucial ingredients of the proof. They are novel arguments which do not seem to have appeared in the literature, and are exactly our most original and genuine contributions; the prototype of these ideas first appeared in our previous investigation \cite{Bin-1}; see also \cite{Bin-3}.
We do hope our new method can inspire more techniques in the field.
\end{remark}

\section*{Acknowledgements}
The author would like to thank Professor Siyuan Lu for fruitful discussions on the subject and for pointing out several lethal mistakes in the initial attempts. Moreover, the paper \cite{SY} of Shankar-Yuan was brought up to the author's attention by Lu and it was Lu who told the author that one may try to apply their almost-Jacobi inequality to obtain Theorem \ref{k=2} in dimension four. This work is part of the author's Ph.D. thesis at The Chinese University of Hong Kong. The author is greatly indebted to Professor Man-Chun Lee and Professor Xiaolu Tan for all the help and encouragement. The author would also like to thank Professor Guohuan Qiu for kind comments on this work and for sharing insights about the Jacobi inequality.

\bibliography{refs}

\end{document}